\documentclass{amsart}
\usepackage{amsmath, amsthm, amssymb, amsfonts, amsthm,tikz-cd,enumerate}
\usepackage{amsxtra, amscd, graphicx,epic,eepic}
\usepackage[all,cmtip]{xy}
\usepackage{mathrsfs}
\usepackage{color}
\usepackage{soul}
\usepackage{geometry}
\newtheorem{theorem}{Theorem}
\newtheorem*{theorem*}{Theorem}

\newtheorem{definition}{Definition}
\newtheorem{algorithm}[theorem]{Algorithm}
\newtheorem{proposition}[theorem]{Proposition}
\newtheorem*{prop*}{Proposition}

\newtheorem*{conj*}{Conjecture}
\newtheorem{lemma}[theorem]{Lemma}
\newtheorem*{lem*}{Lemma}
\newtheorem{cor}[theorem]{Corollary}
\newtheorem*{cor*}{Corollary}

\numberwithin{equation}{section}

\renewcommand{\tilde}[1]{\widetilde{#1}}%

\newcommand{\N}{\mathbb N}
\newcommand{\Q}{\mathbb Q}
\newcommand{\Z}{\mathbb Z}

\newcommand{\R}{\mathbb R}
\renewcommand{\H}{\mathbb{H}}

\newcommand{\MM}{\mathcal M}

\newcommand{\F}{\mathbb{F}}

\newcommand{\llangle}{\langle\! \langle}
\newcommand{\rrangle}{\rangle\! \rangle}

\newcommand{\tOmega}{\tilde{\Omega}}
\newcommand{\SSS}{\mathcal{S}}

\newcommand{\AAA}{\mathcal{A}}
\newcommand{\gplus}{\gamma_{\infty}}
\newcommand{\gneg}{\gamma_{-\infty}}

\newcommand{\balpha}{\bar{\alpha}}
\newcommand{\e}{\epsilon}
\newcommand{\x}{\xi_\gamma}
\newcommand{\y}{\eta_\gamma}
\newcommand{\PSL}{\operatorname{PSL}(2,\Z)}

\title{Geodesic flows and the mother of all continued fractions}
\author{Claire Merriman}

\begin{document}
\maketitle
\begin{abstract}
We extend the Series' \cite{Ser} connection between the modular surface $\MM=\PSL\backslash\H$, cutting sequences, and regular continued fractions to the slow converging Lehner and Farey continued fractions with digits $(1,+1)$ and $(2,-1)$ in the notation used for the Lehner continued fractions. We also introduce an alternative insertion and singularization algorithm for Farey expansions and other non-semiregular continued fractions, and an alternative dual expansion to the Farey expansions so that $\frac{dxdy}{(1+xy)^2}$ is invariant under the natural extension map.

\end{abstract}
\section{Introduction}

The connection between the geodesics on the modular surface $\MM=\PSL\backslash\H$ and ergodic theory allows us to use geometry to prove dynamic and number theoretic properties of the continued fraction map and to use continued fractions to classify geodesics on the modular surface.
Series \cite{Ser} established explicit connections between a subset of the geodesic flow on $\H$, a geodesic coding by cutting sequences, and regular continued fraction dynamics using a cross section of the geodesic flow on $T_1\MM$. We apply the same geodesic coding, but different cross section, to describe the slow Lehner and Farey expansions. Lehner \cite{Leh} introduced the Lehner expansions on $[1,2)$ using the farthest integer map, and Dajani and Kraaikamp \cite{DK2} introduced the Farey expansions as a dual continued fraction expansion to the Lehner expansions.

The group $\PSL$ acts on the upper half plane $\H=\{x+iy:y>0\}\cup\{\infty\}$ by M\"obius transformations $\left(
\begin{smallmatrix}
 a&b\\c&d
\end{smallmatrix}\right)z=\frac{az+b}{cz+d}$ which preserve the hyperbolic metric $ds^2=\frac{dx^2+dy^2}{y^2}$.  We also define the Farey tessellation $\F$ made of ideal hyperbolic triangles whose edges are the images of $i\R$ under the $\PSL$ action.

\begin{figure}
 \includegraphics[width=\textwidth]{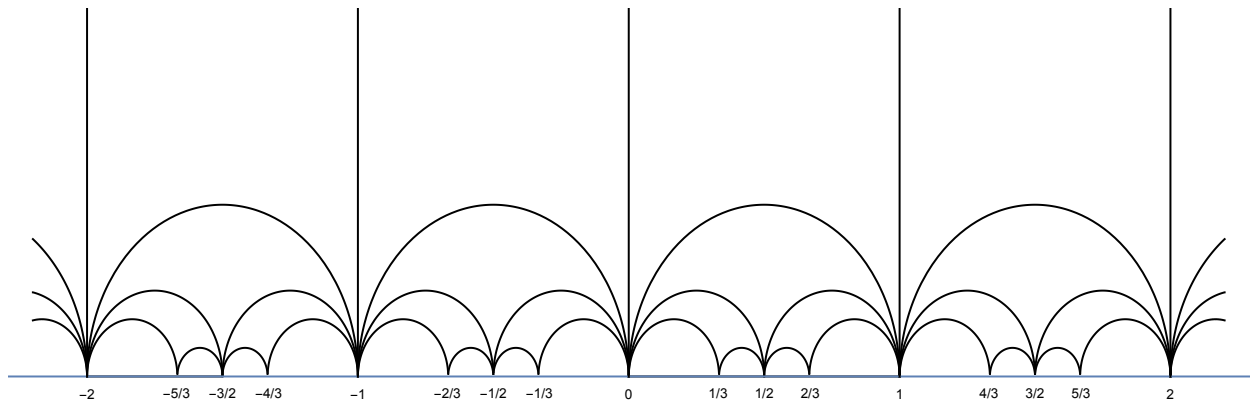}
 \caption{The edges of the ideal triangles are images of $i\R$ under the $\PSL$ action on $\H$. The edges connect two rational numbers  if and only if they are adjacent in some Farey sequence $F_n=\{\frac{p}{q}:0\leq q\leq n\}$. Here is the Farey tessellation up to $n=3$.}
\end{figure}

The tessellation $\F$ breaks a geodesic $\gamma$ on $\H$ into segments with one segment for each triangle the geodesic crosses. The segment crosses two sides of the triangle, and we label it $L$ or $R$, according to whether the vertex shared by the sides is to the left or right of the geodesic. This labeling is invariant under the $\PSL$ action and corresponds to whether the projection of $\gamma$ to $\PSL\backslash\H$ wraps counterclockwise or clockwise around the cusp of $\MM$.  These geodesics $\gamma$ are lifts of geodesics $\bar\gamma$ on $\MM$, where $\overline{\gamma}$ are uniquely determined by
infinite two-sided cutting sequences $\dots L^{n_{-1}}R^{n_0}L^{n_1}\dots$.  This sequence of positive integers $(n_i)_{i=-\infty}^\infty$ gives the regular continued fraction expansion of the forward and backwards endpoints of some lift $\gamma$
\begin{equation}
 \gplus= n_0+\cfrac{1}{n_1+\cfrac{1}{n_2+\dots}}=[n_0;n_1,n_2,\dots], \gneg=\displaystyle \cfrac{-1}{n_{-1}+\cfrac{1}{n_{-2}+\dots}}=-[n_{-1},n_{-2},n_{-3},\dots].
\end{equation}

 Shifting along the cutting sequences corresponds to a two-fold cover of the natural extension of the regular Gauss map $\overline T$ on $[0, 1)^2$ defined by 
 \[
 \overline T([n_0;n_1,n_2,\dots],[n_{-1},n_{-2},\dots])=([n_1;n_2,n_3,\dots],[n_0,n_{-1},n_{-2},\dots]).
 \]
The slow down of this shift corresponds to the natural extension of the Lehner Gauss map.
Some different approaches for coding the geodesic flow on $T_1\MM$ were considered by Arnoux in \cite{Arn}, Katok and Ugarcovici in \cite{KU1}, and Moeckel in \cite{Moe}. Heersink \cite{Hee} also considered the geodesic flow on $\MM$ to classify the distribution of periodic points of the Farey tent map, which is conjugate to the Lehner Gauss map by $x\mapsto x+1$. 

This paper describes the Lehner expansions and their dual continued fraction expansions using the geodesic flow on $T_1 \MM$ as follows:

\begin{theorem}\label{thm1}[Section \ref{cut}] Our classification of geodesics $\bar\gamma$ on $T_1\MM$ with cutting sequence $\dots L^{n_{-1}}R^{n_0}L^{n_1}\dots$ depends on whether or not $n_0=1$. In the one line notation given in \eqref{lehnerdef} and \eqref{fareydef}, if $n_0=1$, $\bar\gamma$ has a lift on $\H$ with forward endpoint 
 \[\gplus =[\![(2,-1)^{n_1-1}(1,+1)(2,-1)^{n_2-1}(1,+1)\dots]\!]
 \]
 and backwards endpoint 
 \[\gneg =\llangle(+1/1)(-1/2)^{n_{-1}}(+1/1)(-1/2
)^{n_{-2}-1}\dots\rrangle
 \] when $n_{-1}\geq 2$ and
  \[\gneg  =\llangle(+1/1)(+1/1)(-1/2)^{n_{-1}-1}(+1/1)(-1/2)^{n_{-2}-1}\dots\rrangle
 \] when $n_{-1}=1$. 
 When $n_0>1$, $\bar\gamma$ has a lift on $\H$ with endpoints
 \[\gplus =[\![(2,-1)^{n_1-1}(1,+1)(2,-1)^{n_2-1}(1,+1)\dots]\!],\]\[
\gneg =\llangle(+1/1)(-1/2)^{n_{-1}}(+1/1)(-1/2)^{n_{-2}-1}\dots\rrangle.
 \]
\end{theorem}

Dajani and Kraaikamp \cite{DK2} call the Lehner expansions the \emph{mother of all semiregular continued fractions} because insertion and singularization algorithms described in Section \ref{leh} take the Lehner expansions to any other semiregular continued fraction. In Section \ref{far_insert}, we define an alternate insertion and singularization algorithm that converts between the Farey expansions and non-semiregular continued fractions.  We also describe the geodesic flow on the modular surface and its tangent space in Section \ref{surface} using a nonstandard fundamental and a different cross section of the geodesic flow on $T_1 \MM$. The connection between Lehner and Farey expansions and the cutting sequences is in Section \ref{cut}. We describe several results that follow from this construction in Section \ref{application} and give an alternate dual continued fraction expansion in Section \ref{alt_leh}.

\section{Lehner expansions}\label{leh}

Lehner \cite{Leh} showed that every irrational number $x\in[1,2)$ has a unique continued fraction expansion of the form 
\begin{equation}\label{lehnerdef}x=a_0+\cfrac{\e_0}{a_1+\cfrac{\e_1}{a_2+\cfrac{\e_2}{\dots}}}=[\![(a_0,\e_0)(a_1,\e_1)(a_2,\e_2)\dots]\!],
\end{equation}
where $(a_i,\e_i)\in\{(2,-1),(1,+1)\}$. Every rational number has two or four finite expansions corresponding to the fact that $2-\tfrac{1}{1}=1$ and $2-\frac{1}{1+\tfrac{1}{1}}=2-\frac{1}{2}=1+\frac{1}{2}=1+\frac{1}{1+\tfrac{1}{1}}$. Lehner \cite{Leh} requires that the final $\frac{\e_{i-1}}{a_i}=\frac{+1}{2}$ for unique finite expansions of rational numbers. These continued fractions are generated by the transformation 
\[L:[1,2)\to[1,2),\quad
L(x):=
\begin{cases}
 \frac{1}{2-x} &\textnormal{if } x\in[1,\frac{3}{2}),\\
  \frac{1}{x-1} & \textnormal{it }x\in[\frac{3}{2},2)
\end{cases}
\]
called the \emph{Lehner Gauss map}.
We get $(a_i,\e_i)=
\begin{cases}
(2,-1) & \textnormal{if } L^{i}(x)\in[1,\frac{3}{2}),\\
 (1,+1) & \textnormal{if } L^{i}(x)\in[\frac{3}{2},2).
\end{cases}$

The map $L$ is conjugate to the Farey map \[\tau:[0,1)\to[0,1),\quad
\tau(x)=\begin{cases}
 \frac{x}{1-x} & \textnormal{if } x\in[0,\frac{1}{2})\\
  \frac{1-x}{x} & \textnormal{if } x\in[\frac{1}{2},1),
\end{cases}
\]
as $L=T\circ \tau\circ T^{-1}$ where $T(x)=x+1$.
Ito \cite{Ito} showed that $\tau$ is ergodic and has $\sigma$-finite invariant measure $\frac{dx}{x}$. As a result, $L$ is ergodic with $\sigma$-finite invariant measure $\frac{dx}{x-1}$.

Dajani and Kraaikamp \cite{DK2} describe the insertion and singularization algorithm to convert from a regular continued fraction expansion 
\[x=1+\cfrac{1}{n_1+\cfrac{1}{n_2+\dots}}=[1;n_1,n_2,\dots], \quad n_i\in\N,
\] to the corresponding Lehner continued fraction. By repeatedly applying this algorithm, we get
\begin{equation}\label{leh_insert}
[1;n_1,n_2,\dots]=[\![(2,-1)^{n_1-1}(1,+1)(2,-1)^{n_2-1}(1,+1)\dots]\!],
\end{equation}
where $(2,-1)^t$ means the digit $(2,-1)$ appears $t$ times, and $(2,-1)^0$ is the empty word. Note that when $n_1=1$, we get 
$[1;1,n_2,\dots]=
[\![(1,+1)(2,-1)^{n_2-1}(1,+1)\dots]\!]$. Looking at the finite expansions again, this rule gives
\begin{equation*}
\begin{split}
 1+\frac{1}{n}&= [\![(2,-1)^{n-1}(1,+1)]\!]= [\![(2,-1)^{n-2}(1,+1)(1,+1)]\!]=1+\cfrac{1}{n-1+\frac{1}{1}},\quad n\geq 2,
 \end{split}
\end{equation*}
corresponding to $[\![(2,-1)(1,+1)]\!]=2-\frac{1}{1+1}=1+\frac{1}{1+1}=[\![(1,+1)(1,+1)]\!]$. However, we can replace the final $(1,+1)(1,+1)$ with $(2,-1)$, producing four possible expansions.

\section{Farey continued fractions}\label{far_insert}

Dajani and Kraaikamp \cite{DK2} call the dual continued fraction expansion the Farey continued fraction expansion. These continued fractions have the form 
\begin{equation}
\label{fareydef}y=\cfrac{f_0}{b_0+\cfrac{f_1}{b_1+\dots}}=\llangle(f_0/b_0)(f_1/b_1)(f_2/b_2)\dots\rrangle
\end{equation}
where $(f_i/b_i)\in\{(-1/2),(+1/1)\}$ for $y\in[-1,\infty)$. The Gauss map for the Farey expansions is \[F(x):=
\begin{cases}
 \frac{-1}{x} -2&\textnormal{if }  x\in\left[-1,0\right)\\
 0 &\textnormal{if }  x=0\\
  \frac{1}{x}-1 &\textnormal{if }  x\in\left(0,\infty\right),
\end{cases}
\] 
where $(f_i/b_i)=\begin{cases}
(-1/2) & \textnormal{if } F^{i}(x)\in [-1,0),\\
(+1/1) & \textnormal{if } F^{i}(x)\in (0,\infty).
\end{cases}$ 

Dajani and Kraaikamp \cite{DK2} show that $F$ is ergodic with $\sigma$-finite invariant measure with density $\frac{1}{x+1}-\frac{1}{x+2}=\frac{1}{(x+1)(x+2)}$.
Converting from regular to Lehner continued fractions uses an insertion algorithm based on the identity
\begin{equation}\label{insert}
 A+\frac{\e}{B+\xi}=A+\e+\cfrac{-\e}{1+\cfrac{1}{B-1+\xi}},
\end{equation}
as described by Kraaikamp in \cite{Kr}. Kraaikamp's insertion and singularization algorithm works for semi-regular continued fractions,  where for every $b_i+\frac{f_{i+1}}{b_{i+1}+\dots}, b_i+f_{i+1}\geq 2$.  In other words, this algorithm does not allow $\frac{1}{1-\xi}$, which occurs in the Farey continued fractions when $(f_i/b_i)(f_{i+1}/b_{i+1})=(+1/1)(-1/2)$. Thus, we need to define a new insertion algorithm based on the identity 
\begin{equation}\label{alt_insert}
 A+\frac{\e}{B+\xi}=A-\e+\cfrac{\e}{1-\cfrac{1}{B+1+\xi}}.
\end{equation}

We use this new insertion and singularization algorithm to generate the Farey expansions. In the process of generating the Farey expansions, we will get intermediate steps that are not Farey expansions.  
\begin{algorithm}
 Let $-1<y$ with regular continued fraction expansion $\pm[n_0;n_{-1},n_{-2},\dots]$, where $n_0=0$ when $-1<y<1$. Then the following algorithm produces the Farey expansion of $y$.
 
\begin{enumerate}
\item Let $-1<y<0$. If $n_{-1}=2$, move to $n_{-2}$. 
\begin{enumerate}
\item If $n_{-1}=1$, use identity \eqref{insert}. If $n_{-2}=1,$ we start on the right hand side of \eqref{insert} with $A+\e=1, -\e=+1,$ and $B-1=n_{-3}$ to remove the digit $(+1/1).$ This yields
\[1+\cfrac{1}{1+\cfrac{1}{n_{-3}+\dots}}=2-\cfrac{1}{n_{-3}+1+\dots}.\] 
If $n_{-2}>1$, we start on the left hand side of \eqref{insert} with $A=1, \e=+1,$ and $B=n_{-2}$ \[1+\cfrac{1}{n_{-2}+\dots}=2-\cfrac{1}{1+\cfrac{1}{n_{-2}-1+\dots}},\]
which gives the first digit $(-1/2)$ and begins the process of converting $(1/n_{-2})$ to a Farey digit.
\item If $n_{-1}>2$, use the identity \eqref{alt_insert} with $A=n_{-1}, \e=+1$, and $B=n_{-2}$ to get \[n_{-1}+\cfrac{1}{n_{-2}+\dots}=n_{-1}-1+\cfrac{1}{1-\cfrac{1}{n_{-2}+1+\dots}}.\]
Applying the identity \eqref{alt_insert} repeatedly gives \[-[0;n_{-1},n_{-2},\dots]=\llangle(-1/2)(1/1)(-1/2)^{n_{-1}-2}(-1/(n_{-2}+1))\dots\rrangle.\] 
\end{enumerate}

\item Let $0<y$ and $k$ be the first index where $n_k>1$. Use identity  \eqref{alt_insert} to get \[n_{k}+\cfrac{1}{n_{k-1}+\dots}=n_{k}-1+\cfrac{1}{1-\cfrac{1}{n_{k-1}+1+\dots}}.\]
Applying the identity repeatedly gives 
\begin{equation*}
\begin{split}
 [0;1,...1,n_{k},n_{k-1},\dots]=\llangle(1/1)^{k+1}(-1/2)^{n_{k}-1}(-1/(n_{k-1}+1))\dots\rrangle,\\
 [1;1,...1,n_{k},n_{k-1},\dots]=1+\llangle(1/1)^{k}(-1/2)^{n_{k}-1}(-1/(n_{k-1}+1))\dots\rrangle,
\end{split}
\end{equation*}
since there were already $k-1$ copies of $(1/1)$ and we inserted two more. 

\item For $y>1$, apply (2) then use identity \ref{alt_insert} to get \[1+\cfrac{1}{1+z}=\cfrac{1}{1- \cfrac{1}{2+z.}}\]
\end{enumerate}
Repeat with the next digit that is not $(1/1)$ or $(-1/2)$.
\end{algorithm}

Summarizing, we have two cases to consider, corresponding to the intervals $(-1,0),$ $(0,\infty)$.

\begin{description}
 \item[(a) $0<y$] $y=n_{0}+\cfrac{1}{n_{-1}+\cfrac{1}{n_{-2}+\dots}}=\llangle(1/1)(-1/2)^{n_{0}}(1/1)(-1/2)^{n_{-1}-1}\dots\rrangle$, where $n_0=0$ when $0<y<1$.

 \item[(b) $-1<y<0$] 
 \[\hspace{-3em}y=
 \cfrac{-1}{n_{-1}+\cfrac{1}{n_{-2}+\dots}} \\=
\begin{cases}
  \llangle(-1/2)^{n_{-2}+1}(1/1)(-1/2)^{n_{-3}-1}\dots\rrangle & \textnormal{if }n_{-1}=1\\
  \llangle(-1/2)(1/1)(-1/2)^{n_{-1}-2}(1/1)(-1/2)^{n_{-2}-1}\dots\rrangle& \textnormal{if } n_{-1}>1

\end{cases}\]

\end{description}

\noindent where $(-1/2)^t$ means the digit $(-1/2)$ appears $t$ times, and $(-1/2)^0$ is the empty word, as in the Lehner case. An irrational number $x$ has an eventually periodic regular continued fraction expansion if and only if it is the root of a rational quadratic equation, called a quadratic irrational. Combining this fact with the above algorithm gives

\begin{cor}\label{farey_quad}
If $x$ is a quadratic irrational, then its Farey expansion is eventually periodic.
\end{cor}
In Section \ref{application}, we prove that the Farey expansion of an irrational number $x$ is eventually periodic if and only if $x$ is a quadratic irrational. Note that we need to specify irrational, since $-1=\llangle\overline{(-1/2)}\rrangle$.

The Farey continued fractions allow us to construct an invertible natural extension of the Lehner map $L$, $\mathcal{L}:[1,2)\times[-1,\infty)\to [1,2)\times[-1,\infty)$ defined by
\[\mathcal{L}(x,y)=\left( \frac{\e_0(x)}{x-a_0(x)},  \frac{\e_0(x)}{y+a_0(x)}\right)=
\begin{cases}
( \frac{-1}{x-2},  \frac{-1}{y+2})&  \textnormal{if }x\in[1,\frac{3}{2}),\\
( \frac{1}{x-1},  \frac{1}{y+1})&   \textnormal{if } x\in[\frac{3}{2},2).
\end{cases}
\] On the continued fraction expansions, $\mathcal{L}$ acts as the shift map 
\[\mathcal{L}\left([\![(a_0,\e_0)(a_1,\e_1)\dots]\!], \llangle(b_0,f_0)(b_1,f_1)\dots\rrangle\right)=\left([\![(a_1,\e_1)(a_2,\e_2)\dots]\!], \llangle(\e_0/a_0)(f_0/b_0)\dots\rrangle\right).
\] Dajani and Kraaikamp \cite{DK2} also showed that the invariant measure for $\mathcal{L}$ has density $\frac{1}{(x+y)^{2}}$.

It will be helpful to define $\Omega=[1,2)\times[-1,\infty)$ and consider the extension $\tilde{\mathcal{L}}$ of $\mathcal{L}$ to $\tilde{\Omega}:=\Omega\times\{-1,1\}$ defined by 
\[\tilde{\mathcal{L}}(x,y,\e):=(\mathcal{L}(x,y),-\e_0(x)\e).
\]

\section{Cutting Sequences and $\MM$}\label{surface}
\subsection{The group $\PSL$ and $\MM =\PSL\backslash\H$} 

We consider the group generated by $S(z)= \frac{2z-3}{z-1}=2-\frac{1}{z-1}$ and $T(z)=z+1.$ Since $T^{-2}ST(z)=\frac{-1}{z}$, this group is $\PSL$. We take
\begin{equation*}
{\mathfrak F}  =\left\{ z\in \H :1\leqslant\operatorname{Re}z\leqslant2,  \left| z-1 \right| \geqslant 1,  \left| z-2 \right| \geqslant 1\right\}
\end{equation*}
as the fundamental domain for $\MM=\PSL\backslash\H$, depicted in the image on the left of Figure \ref{domains}. This fundamental domain comes from applying $T$ to the fundamental domain Series \cite{Ser} used to describe the regular continued fractions with cutting sequences, which agrees with the fact that $L(x)$ is conjugate to the slow down of the regular Gauss map by $x+1$.
The element $S$ rotates $\H$ about $\frac{3+1\sqrt{2}}{2}$, taking the hyperbolic geodesic $[1+i, \tfrac{3+i\sqrt{3}}{2}]$ to the hyperbolic geodesic $[ \tfrac{3+i\sqrt{3}}{2},2+i]$ and $T$ takes $[1+i,\infty]$ to $[2+i,\infty]$.
The resulting quotient space $\MM=\pi(\H)$ is the usual modular surface, homeomorphic to a sphere with a cusp at $\pi(\infty)$, and cone points at $\pi(\tfrac{3+i\sqrt{3}}{2})$ and $\pi(1+i)$.

We take \begin{equation}\begin{split}\Delta&=ST(\mathfrak{F} \cup S(\mathfrak{F} )\cup S^2(\mathfrak{F} ))=\left\{ z\in \H :\left| z-\tfrac{3}{2}\right| \leqslant 1,  \left| z-\tfrac{5}{4}\right| \geqslant \tfrac{1}{4}, \left| z-\tfrac{7}{4}\right| \geqslant \tfrac{1}{4}\right\}\end{split}\end{equation}
to be the fundamental cell of a tessellation of $\H$ shown on the right of Figure \ref{domains}. It follows from the fact that $\PSL(i\R)$ gives the Farey tessellation (or B\'ezout's identity) that:

\begin{figure}
    \centering
    \includegraphics[width=.4\textwidth]{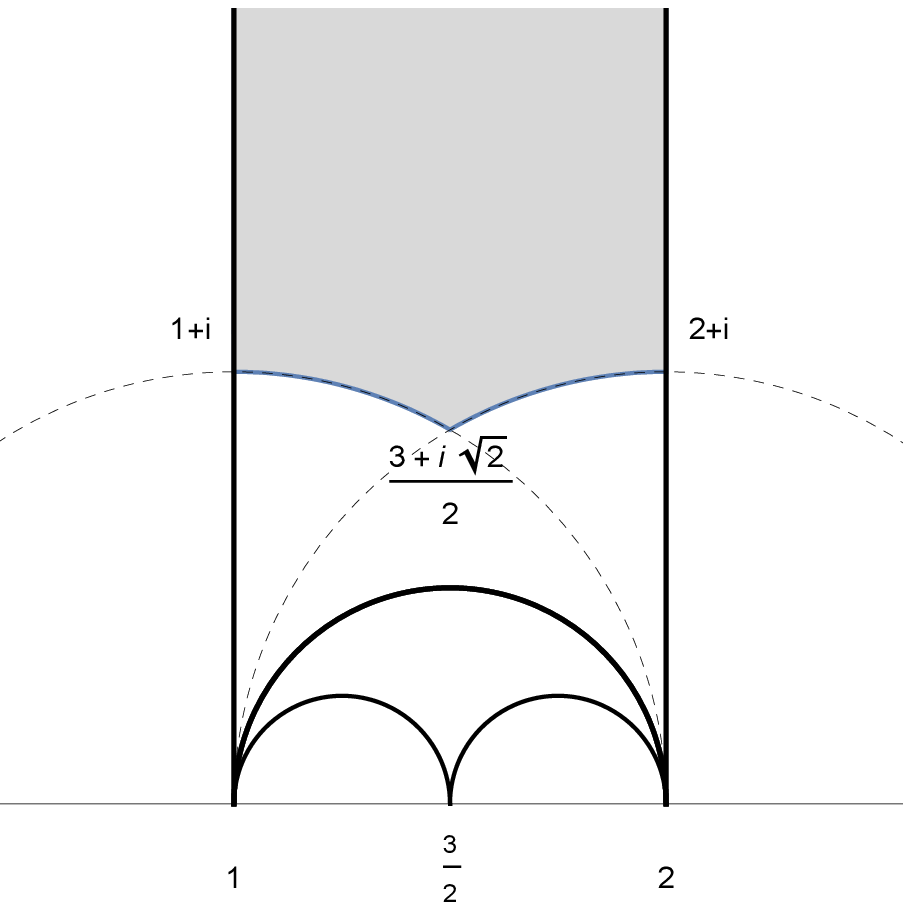}\qquad
    \includegraphics[width=.4\textwidth]{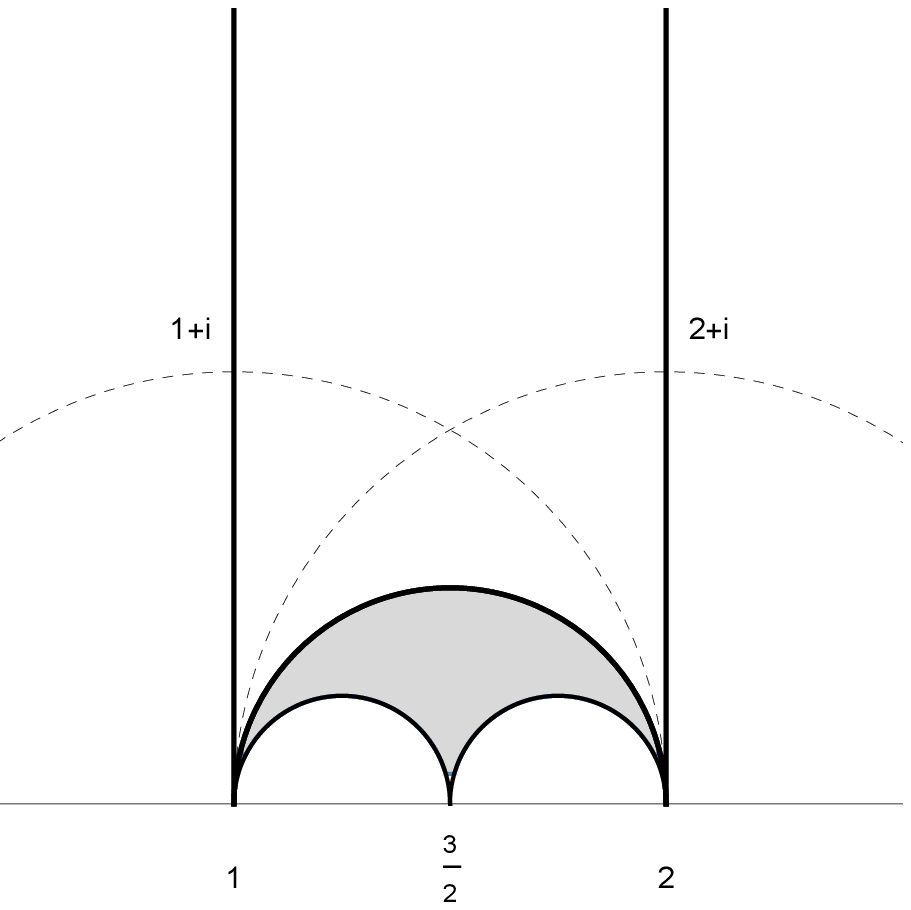}
    \caption{The fundamental domain ${\mathfrak F}$ is shown in grey in the image on the left. The fundamental Farey cell $\Delta$ is shown in grey on the right.}
    \label{domains}
\end{figure}

\begin{lemma}\label{orbit_0}
  $\PSL0=\Q\cup\{\infty\}$.
\end{lemma}

We denote by $\AAA $ the set of geodesics $\gamma$ in $\H$ with endpoints satisfying
\begin{equation*}
(\gplus,\gneg) \in
\SSS :=\left( (1,2) \times (-\infty,1) \right) \cup \left( (-2,-1) \times (-1,\infty)\right).
\end{equation*}

\begin{lemma}\label{geo_lifts_leh}
Every geodesic $\bar{\gamma}$ on $\MM$ lifts to $\H$ to a geodesic $\gamma \in \AAA $.
\end{lemma}
\begin{proof}
The action $T$ glues the line $[1+i,\infty]$ to $[2+i,\infty]$, and $S$ glues $[1+i, \tfrac{3+i\sqrt{3}}{2}]$ to $[\tfrac{3+i\sqrt{3}}{2},2+i]$. Thus, without loss of generality, we can take $\bar{\gamma}$ to be a negatively oriented geodesic arc in $\mathfrak{F} $ in one of the following cases:
\begin{enumerate}
 \item connecting $[1+i,\infty]$ to $[1+i, \tfrac{3+i\sqrt{3}}{2}]$
  \item connecting $[1+i,\infty]$ to $[\tfrac{3+i\sqrt{3}}{2},2+i]$
   \item connecting $[1+i, \tfrac{3+i\sqrt{3}}{2}]$ to $[2+i,\infty]$ 
  \item connecting $[\tfrac{3+i\sqrt{3}}{2},2+i]$ to $[2+i,\infty]$.
\end{enumerate}
In the first case, we consider all geodesics that cross both the geodesic $[1+i,\infty]$ and the geodesic $[1+i, \tfrac{3+i\sqrt{3}}{2}]$. The extremes of this set are the geodesics $[1,\infty]$ and $[0,2]$. Thus, we see that $\gneg<0<1<\gplus<2$. Another way to see this is that any geodesic crossing $[1+i,\infty]$ must have an endpoint greater than $1$. In order to cross $[1+i, \tfrac{3+i\sqrt{3}}{2}]$, the geodesic cannot cross $[2+i, \infty]$. Thus, the endpoint must be less than $2$. Since hyperbolic geodesics are Euclidean semicircles centered on the real axis, and the geodesic must cross above $1+i$, the radius is at least 1. Thus, the other endpoint must be less than $0$.

The transformation $a(z)=T^{-3}ST(z)=-1-\frac{1}{z}$ gives $-2<a(\gplus)<-\frac{3}{2}<1<a(\gneg)$. Thus, $a(\gamma)$ gives a lift of $\bar{\gamma}$ in $\AAA $. 

In the second case, $\gneg<1<\frac{3}{2}<\gplus<\frac{5}{2}$. When $\gplus<2,$ $\gamma$ is in $\AAA $, otherwise $T^{-1}(\gneg)<0<1<T^{-1}(\gplus)<\frac{3}{2}$ defines a lift of $\bar{\gamma}$ in $\AAA$. For the third, $0<\gneg<2<3<\gplus$, which is moved to $\AAA $ by $T^{-\lfloor\gplus\rfloor+1}(z)=z-\lfloor\gplus\rfloor+1$. Finally, in the fourth case, $1<\gneg<2<3<\gplus$ again is moved to $\AAA $ by $T^{-\lfloor\gplus\rfloor+1}(z)=z-\lfloor\gplus\rfloor+1$.
\end{proof}

\subsection{Cutting sequences and continued fraction expansions}\label{lehcut}

The coding of geodesics $\bar{\gamma}$ on $\MM$ is the same as the coding for the regular continued fractions. For the Lehner and Farey continued fractions, we choose the lift $\gamma\in\AAA $ instead of $\pm\big((1,\infty)\times(-1,0)\big)$.  An oriented geodesic $\gamma$ in $\H$ is cut into segments as it crosses triangles in the Farey tessellation $\F$. Each segment of the geodesic crosses two sides of a triangle in the tessellation. If the vertex where the two sides meet is on the left, we label the segment $L$, if it is on the right we label it $R$. This can be thought of as $\gamma$ turning left or right as it crosses the triangle.

\begin{proposition}\label{rcf_lifts}\cite[Section 1.2]{Ser}
 Every geodesic $\bar\gamma$ on $\MM$ other than the line from $\pi(\infty)$ to $\pi(i)$ to $\pi(\infty)$ lifts to a geodesic $\gamma$ in $\H$. These geodesics have cutting sequence $\dots L^{n_{-1}}R^{n_0}L^{n_1}\dots$ described above. Since different lifts of $\bar\gamma$ differ by covering translations which leave the Farey tessellation invariant and preserve orientation, the labels of a segment and hence the cutting sequences are independent of the lift chosen.
\end{proposition}

Since any M\"obius transformations preserve the hyperbolic distance, a map $M:\SSS\to\SSS$ induces a map on $\AAA$. Hyperbolic geodesics are uniquely determined by their endpoints, so we can describe maps from $\AAA$ to itself by the action on the endpoints. Thus, we will use both $M(\gplus,\gneg)$ and $M(\gamma)$. Consider the set $X$ of unit tangent vectors based on the projection of the geodesics $\pm[1,2]$ to $\pi(\pm[1,2])$ that point along geodesics in $\pi(\AAA)$. We will use $X$ as a cross-section of the geodesic flow on the unit tangent bundle $T_1 \MM$. 
However, we can identify the tangent vectors with their base points $\x$, since both the vector and $\x$ are uniquely determined by $\gamma$.  

In the case of Lehner continued fractions, the first digit of the continued fraction expansion of $\gplus$ is determined by whether $\gamma$ turns left or right when crossing the triangle with vertices $\{1,\frac{3}{2},2\}$ (for $1\leq\gplus<2$) or  $\{-1,-\frac{3}{2},-2\}$ (for $-2<\gplus\leq-1$). 
To every geodesic $\gamma \in \AAA $ we associate the positively
oriented geodesic arc $[\x,\eta_\gamma]$, where
\begin{equation*}
\xi_\gamma:=\begin{cases}
\gamma \cap [1,2] & \mbox{\rm if $1\leq\gplus<2$} \\
\gamma \cap [-2,-1] & \mbox{\rm if $-2<\gplus\leq-1$}
\end{cases}
\quad \mbox{\rm and} \quad
\eta_\gamma:=\begin{cases}
\gamma \cap [a_0+\e_0,\tfrac{3}{2}] & \mbox{\rm if $1\leq\gplus<2$} \\
\gamma \cap [-a_0-\e_0,-\tfrac{3}{2}] & \mbox{\rm if $-2<\gplus\leq-1$},
\end{cases}
\end{equation*}
with $(a_0,\e_0)=(a_0(\gplus),\e_0(\gplus))$.
That is, $\x$ is where the geodesic enters the cell and $\eta_\gamma$ is where $\gamma$ exits the cell. This construction gives an alternate definition of $X $ as the collection of unit tangent vectors based at $\pi(\x)$ pointing along $\pi(\gamma)$, and $\pi(\y)$ as the base point of the first return of the geodesic flow to the cross section $X $. The cross section of the geodesic flow for the regular continued fractions is the set of unit tangent vectors based at $\pi(i\R)$ pointing along $\pi(\gamma)$ whose cutting sequence changes from $L$ to $R$ or from $R$ to $L$ at $\pi(i\R$). Thus, the regular continued fraction case considers geodesics with endpoints in  $\pm\big((1,\infty)\times(-1,0)\big)$ and unit tangent vectors based on $i\R$ which point along $\gamma$. Then the first return to the cross section of the geodesic flow lifts to the next place where the cutting sequence changes type and the last place where $\gamma$ crosses a vertical side of the Farey tessellation. Thus, the Lehner and Farey expansions come from a slow down of the regular continued fractions.

We consider the geodesic $\gamma\in\AAA$ with $\gplus=\e[\![(a_0,\e_0)(a_1,\e_1),\dots]\!]$, $\gneg=-\e\llangle
(\e_{-1}/a_{-1})(\e_{-2}/a_{-2}) \dots \rrangle$, where $\e=\operatorname{sign}\gplus$.
For a point $z\in\gamma,$ we define the map $\rho(z):=\frac{1}{\e a_0 -z}$ and denote the induced map on $\SSS$ as $\bar\rho$. Thus, 
\begin{equation}\label{rholeh}
\begin{split}\bar{\rho}\left(\gplus,\gneg\right)&=\left(\frac{1}{\e a_0-\gplus},\frac{1}{\e a_0-\gneg} \right)
=\left( -\e \e_0[\![ (a_1,\e_1)(a_2,\e_2)\dots ]\!], \e \e_0
\llangle (\e_0/a_0)(\e_{-1}/a_{-1})\dots \rrangle\right).
\end{split}
\end{equation}
Note that $\bar{\rho}$ takes the geodesic arc $\e[1,2]$ to the arc $-\e\e_0[1,\infty]$, and it takes the geodesic $\e[a_0+\e_0,\tfrac{3}{2}]$ to $-\e\e_0[1,2]$.

Since the endpoints of the geodesic also uniquely determine $\x$ and the unit tangent vector pointing along $\gamma$, $\bar\rho$ also induces a map on $X $.

\begin{figure}
 \includegraphics[width=.8\textwidth]{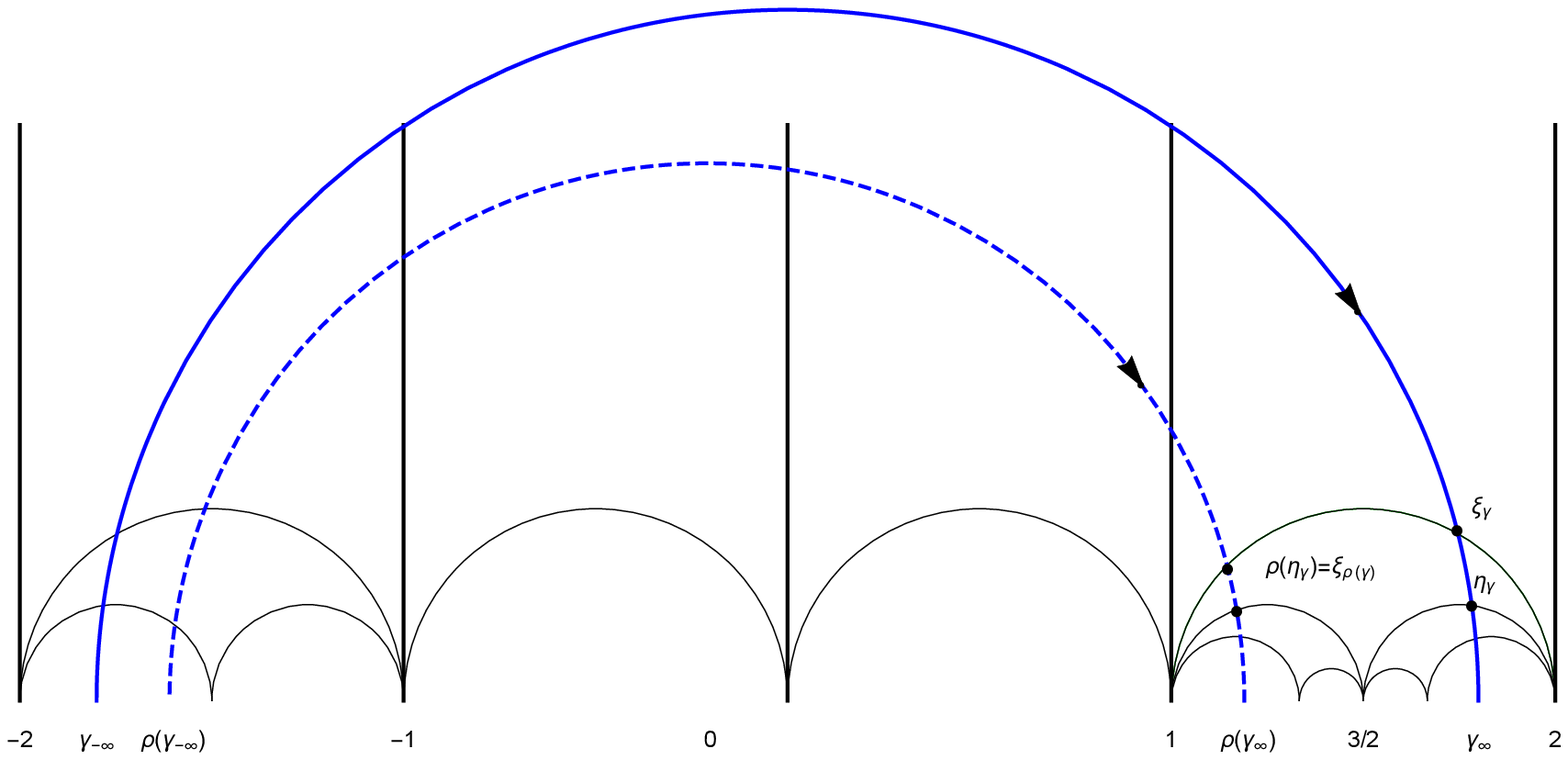}
\caption{$\gamma$ (solid) has cutting sequence $\dots L R L^2 R \x L \eta_\gamma L\dots$, $\bar{\rho}(\gamma)$ (dashed) has cutting sequence $\dots L R L^2 R \xi_{\bar{\rho}(\gamma)} R \eta_{\bar{\rho}(\gamma)}  R^2\dots$}\label{lehner1}
\end{figure}

\begin{figure}
 \includegraphics[width=.8\textwidth]{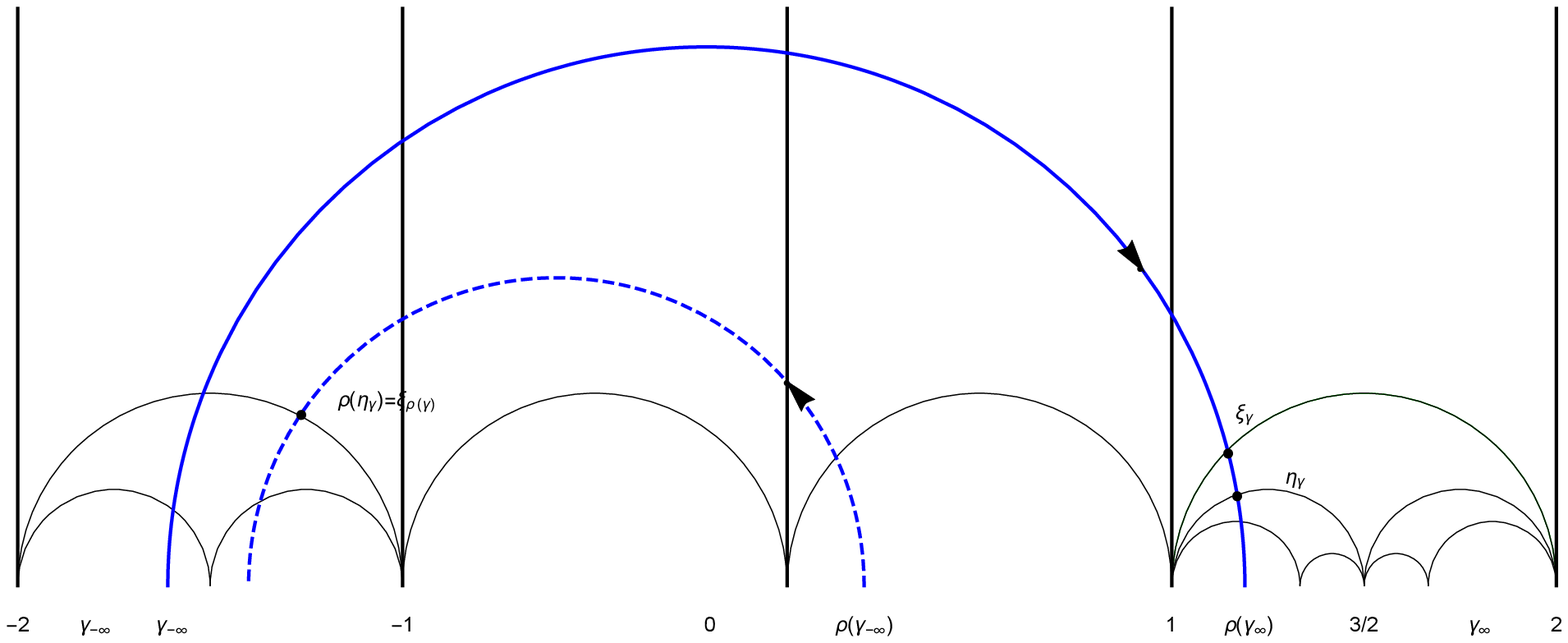}
 \caption{$\gamma$ (solid) has cutting sequence$\dots L R L^2 R \x R \eta_\gamma R^2\dots$, $\bar{\rho}(\gamma)$ (dashed) has cutting sequence $\dots LR L  \xi_{\bar{\rho}(\gamma)} L \eta_{\bar{\rho}(\gamma)L }\dots $}\label{lehner2}

\end{figure}

\begin{theorem}\label{commuteleh}
The map $\bar{\rho }:\SSS  \to \SSS $ is invertible, and the diagram
\begin{equation*}
\begin{tikzcd}[ column sep = large]
 \mathcal{S} \arrow{r}{\overline\rho}\arrow[swap]{d}{J} & \mathcal{S} \arrow{d}{J} \\
 \widetilde{\Omega} \arrow{r}{\tilde{\mathcal{L}}} & \widetilde{\Omega}
\end{tikzcd}
\end{equation*}
commutes, where $J:\SSS  \rightarrow \tOmega$ is the invertible map defined by
\begin{equation*}
J(x,y):=\operatorname{sign}(x) (x,-y,1) = \begin{cases}
(x,-y,1) & \textnormal{ if } x\in\left[1,2\right),\ y<1 \\
(-x,y,-1) & \textnormal{ if } x\in\left(-2,-1\right],\ y>-1 .
\end{cases}
\end{equation*}
The final coordinate of $J$ and $\tilde{\mathcal{L}}$ keeps track of whether $\overline{\rho}$ is orientation reversing.
\end{theorem}

\begin{proof} Let $x= \e [\![ (a_0,\e_0) (a_1,\e_1)\ldots ]\!]$, $y=-\e \llangle (\e_{-1}/a_{-1})(\e_{-2}/a_{-2})\ldots \rrangle$ with
$\e \in\{ \pm 1\}$, so that $(x,y)\in \SSS $. Then we have 
\begin{align*}J\circ \bar{\rho}(x,y)&=J\left(\frac{1}{\e a_0-x},\frac{1}{\e a_0-y}\right)=\left(\frac{-\e \e_0}{\e a_0-x},\frac{\e \e_0}{\e a_0-y},-\e \e_0\right),\\
 \tilde{\mathcal{L}}\circ J(x,y)&=\tilde{\mathcal{L}}(\e x,-\e y, \e)=\left(\frac{\e_0}{\e x-a_0},\frac{\e_0}{-\e y+a_0},-\e \e_0\right).\qedhere\end{align*}
 \end{proof}
 Thus, we have an alternate proof of:
\begin{cor}\cite[Theorem 1]{DK2}\label{Lmeasure}
The measure $\tfrac{dxdy}{(x+y)^2}$ is $\mathcal{L}$ invariant.
\end{cor}
\begin{proof}

We use the fact that the invariant measure for the first return map for a cross section of the geodesic flow on  $T_{1}\H$ is $\tfrac{d\alpha d\beta d\theta}{(\alpha-\beta)^2}$ where $\alpha,\beta\in\R$ denote the endpoints of the geodesic $\gamma(u)$ through $u\in T_1\H$ and $\theta$ is the distance between the base point of $u$ and the midpoint of $\gamma(u)$ \cite{Hopf}. We use $J$ to push forward this measure and project to the $(x,y)$ coordinates to get the invariant measure for $\mathcal{L}$. That is, since $J\circ \bar{\rho}\circ J^{-1}=\tilde{\mathcal{L}}$, the change of variables formula gives that the measure with density 
\[\left|\frac{d(\e x) d(-\e y) }{\big(\e x -(-\e y)\big)^2}\right|=\frac{dxdy}{(x+y)^2}\] 
is $\tilde{\mathcal{L}}$-invariant. Projecting $\tilde{\mathcal{L}}$ acting on $\tilde{\Omega}$ to $\mathcal{L}$ acting on $\Omega$ by $(x,y,\e)\mapsto(x,y)$ yields $\tfrac{dxdy}{(x+y)^2}$ as a $\mathcal{L}$-invariant measure.
\end{proof}

\section{Connection with cutting sequence and regular continued fractions}\label{cut}

Series \cite{Ser} described an explicit relationship between the cutting sequence of a geodesic and the regular continued fractions. The algorithms described in Sections \ref{leh} and \ref{far_insert} allow us to translate from the cutting sequence of the regular continued fraction expansion and the Lehner and Farey continued fraction expansions. By examining the continued fraction expansions of the endpoints, we prove Theorem \ref{thm1}:

\begin{theorem*} Our classification of geodesics $\bar\gamma$ on $T_1\MM$ with cutting sequence $\dots L^{n_{-1}}R^{n_0}L^{n_1}\dots$ depends on whether or not $n_0=1$. In the one line notation given in \eqref{lehnerdef} and \eqref{fareydef}, if $n_0=1$, $\bar\gamma$ has a lift on $\H$ with forward endpoint 
 \[\gplus =[\![(2,-1)^{n_1-1}(1,+1)(2,-1)^{n_2-1}(1,+1)\dots]\!]
 \]
 and backwards endpoint 
 \[\gneg =\llangle(+1/1)(-1/2)^{n_{-1}}(+1/1)(-1/2
)^{n_{-2}-1}\dots\rrangle
 \] when $n_{-1}\geq 2$ and
  \[\gneg  =\llangle(+1/1)(+1/1)(-1/2)^{n_{-1}-1}(+1/2)(-1/2)^{n_{-2}-1}\dots\rrangle
 \] when $n_{-1}=1$. 
 When $n_0>1$, $\bar\gamma$ has a lift on $\H$ with endpoints
 \[\gplus =[\![(2,-1)^{n_1-1}(1,+1)(2,-1)^{n_2-1}(1,+1)\dots]\!],\]\[
\gneg =\llangle(+1/1)(-1/2)^{n_{-1}}(+1/1)(-1/2)^{n_{-2}-1}\dots\rrangle.
 \]
\end{theorem*}

For the Lehner and Farey continued fractions, we read the cutting sequence one letter at a time. If the letter is the same as the previous (letter to the left), the digit is $(2,-1)$, if it is different than the previous letter, the digit is $(1, +1)$.

\subsection{Lehner continued fractions}

We look at the cutting sequence for the Lehner continued fractions. When $\gplus\in\left[1,2\right)$, we have the sequence $\dots R\x R^{n_1-1} L^{n_2}\dots$ and the regular continued fraction expansion $[1;n_1,n_2,\dots]$.

\begin{description}
 \item[(A)] $1\leq\gplus<\tfrac{3}{2}$, then $n_1>1$. The first letter after $\x$ is $R$, which is the same as the previous letter, so the first digit is $(2,-1)$. In fact, each of the $n_1-1$ $R$'s correspond to the digit $(2,-1)$, so the Lehner continued fraction expansion starts $[\![(2,-1)^{n_1-1}\dots]$. Next, we get an $L$ corresponding to $(1,+1)$ followed by $L^{n_2-1}$ corresponding to $(2,-1)^{n_2-1}$. Continuing in this way, we get $\gplus=[\![(2,-1)^{n_1-1}(1,+1)(2,-1)^{n_2-1}(1,+1)\dots]\!]$, as in equation \eqref{leh_insert}.
 
 \item[(B)] $\tfrac{3}{2}<\gplus<2$, then $n_1=1$. The cutting sequence is now $\dots L^{n_{-2}}R^{n_{-1}}L^1R\x L^{n_2}\dots$, and the first letter after $\x$ is different from the previous. Thus, as in equation \eqref{leh_insert}, we get $\gplus=[\![(1,+1)(2,-1)^{n_2-1}(1,+1)(2,-1)^{n_3-1}(1,+1)\dots]\!]$.

\end{description}

For $-2<\gplus<-1$, the same procedure holds, with $L$'s and $R$'s reversed.

\subsection{Farey continued fractions}

We read the Farey continued fraction expansion of $\gneg$ from right to left starting at $\x$. To more easily see the connection to the cases in Section \ref{far_insert}, we consider $\e=-1$ and $\gneg>-1$.

\begin{figure}
 \includegraphics[width=.8\textwidth]{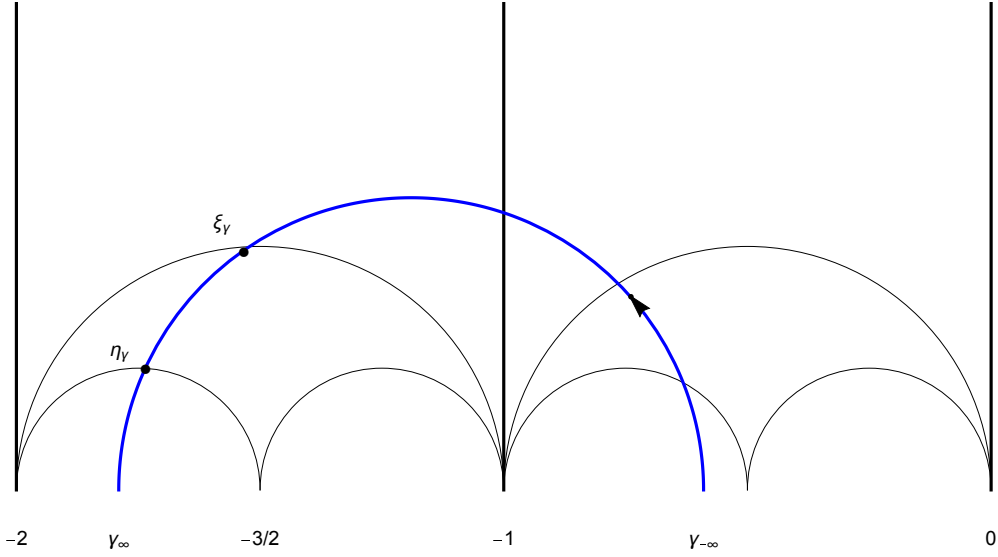}
 \caption{ Geodesic with cutting sequence $\dots L^2 R\x R\y R\dots$}\label{farey1}
 \end{figure}

\begin{description}
 \item[(a)] $1<\gneg$. We get the cutting sequence $\dots L^{n_{-1}}R^{n_{0}+1}L\x\dots$ (the $\e=+1$ case is shown in Figure \ref{lehner1}). Since we are reading from right to left, we start with the final $L$, which differs from the previous letter, so the first digit is $(+1/1)$. Next we have $R^{n_{0}+1}$ corresponding to $(-1/2)^{n_{0}}(+1/1)$. Following this procedure, we get the Lehner expansion $\llangle(+1/1)(-1/2)^{n_{0}}(+1/1)(-1/2)^{n_{-1}-1}\dots\rrangle$ as in Section \ref{far_insert}.
 
 \item[(b)] $0<\gneg<1$. We get the cutting sequence $\dots R^{n_{-2}}L^{n_{-1}}RL\x\dots$, as in the dashed line in Figure \ref{lehner2}. Again, we start with $L$ which differs from the previous letter. Now the preceding letter $R$ also differs from the previous, giving $(+1/1)(+1/1)$ followed by $L^{n_{-1}}$ and $(-1/2)^{n_{-1}-1}(+1/1)$. Continuing this process gives the Lehner expansion $\llangle(+1/1)(+1/1)(-1/2)^{n_{-1}-1}(+1/1)(-1/2)^{n_{-2}-1}\dots\rrangle$.
 
 \item[(c)]  $-1<\gneg<0$. We get the cutting sequence $\dots L^{n_{-1}}R^{n_{0}-1}LL\x\dots$, as in Figure \ref{farey1}. Note the $n_{0}$ is the first digit of the regular continued fraction expansion. It helps to split into the case where $n_{0}=1$ and $n_{0}>1$.
 
\begin{description}
\item[\textnormal{When $n_{0}>1$,}] the first $L$ agrees with the previous letter, so the Lehner expansion starts with $(-1/2)(+1/1)$. Again, $R^{n_{0}}$ corresponds to $(-1/2)^{n_{0}-1}(+1/1)$ and we get the Lehner expansion $\llangle(-1/2)(+1/1)(-1/2)^{n_{0}-2}(+1/1)(-1/2)^{n_{-1}-1}\dots\rrangle$.

 \item[\textnormal{When $n_{0}=1$,}] the cutting sequence is $\dots R^{n_{-2}}L^{n_{-1}+2}\x\dots$, so we start with $n_{-1}+1$ letters that agree with the previous, giving $  \llangle(-1/2)^{n_{-1}+1}(+1/1)(-1/2)^{n_{-2}-1}\dots\rrangle$.

\end{description}
 \end{description}
Again, the case where $\gneg<1$ corresponds to reversing $L$'s and $R$'s.

\section{Applications}\label{application}
In this section, we prove number theoretic results that follow from the use of $\PSL$ to describe the Lehner and Farey expansions.


\begin{lemma}\label{lehner_rho_period}
  If $\alpha=[\![\overline{(a_0,\e_0)\ldots(a_{r-1},\e_{r-1})}]\!]\in(1,2)$, $\beta=-\llangle\overline{(\e_{r-1}/a_{r-1})\ldots(\e_0/a_0)}\rrangle<1$, or \\$\alpha=-[\![\overline{(a_0,\e_0)\ldots(a_{r-1},\e_{r-1})}]\!]\in(-2,-1)$, $\beta=\llangle\overline{(\e_{r-1}/a_{r-1})\ldots(\e_0/a_0)}\rrangle>-1$, then

\begin{enumerate}
 \item \label{rho_r}$\bar{\rho}\ ^r(\alpha,\beta)=\e(-\e_0)(-\e_1)\cdots(-\e_{r-1})(\alpha,\beta)$.
  \item $\bar\rho\ ^{2r}(\alpha,\beta)=(\alpha,\beta)$.
\end{enumerate}
\end{lemma}
\begin{proof}

 Note that 
 \begin{align*}
 \bar\rho(\alpha,\beta)&=(-\e_0\e[\![\overline{(a_1,\e_1)\ldots(a_{r-1},\e_{r-1})(a_0,\e_0)}]]\!], \e_0\epsilon\llangle\overline{(\e_0/a_0)(\e_{r-1}/a_{r-1})\ldots(\e_1/a_1)}\rrangle)\\
 &=(-\e_0\epsilon)([\![\overline{(a_1,\e_1)\ldots(a_{r-1},\e_{r-1})(a_0,\e_0)}]]\!], -\llangle\overline{(\e_0/a_0)(\e_{r-1}/a_{r-1})\ldots(\e_1/a_1)}\rrangle).
 \end{align*} Repeated application gives $\ref{rho_r}$. Continuing to apply $\bar\rho$ gives 
 \begin{align*}
 \bar\rho\ ^{2r}(\alpha,\beta)
 &=\big((-\e_0)(-\e_1)\cdots(-\e_{r-1})\big)^2\epsilon([\![\overline{(a_0,\e_0)\ldots(a_{r-1},\e_{r-1})}]\!], -\llangle\overline{(\e_{r-1}/a_{r-1})\ldots(\e_0/a_0)}\rrangle)\\
 &=(\alpha,\beta).\qedhere
 \end{align*}
 
\end{proof}

\begin{proposition}\label{lehner_pureperiod}
A real number $\alpha \in(1,2)$ has a purely periodic Lehner expansion if and only if
$\alpha$ is a quadratic irrational with $  \bar\alpha <1$. Furthermore, if
\begin{equation}\label{eq_l_pureperiod}
\alpha =[\![ \, \overline{(a_0,\e_0)(a_1,\e_1) \ldots (a_{r-1},\e_{r-1})} \,]\!],
\end{equation}
then
\begin{equation}\label{eq_f_pureperiod}
\balpha=-\llangle \,\overline{(\e_{r-1}/a_{r-1}) \ldots (\e_0/a_0)} \,\rrangle .
\end{equation}
\end{proposition}

\begin{proof}
In one direction, suppose that $\alpha$ is given by \eqref{eq_l_pureperiod}. Consider
the geodesic $\gamma \in \AAA $ with endpoints at $\gplus=\alpha$ and
$\gneg=\beta =- \llangle \,\overline{(\e_{r-1}/a_{r-1}), \ldots, (\e_0/a_0)} \,\rrangle<1$. Lemma \ref{lehner_rho_period} shows that the geodesic $\gamma$ is fixed by $\rho^{2r}$, so it is fixed by some $M=
\left(\begin{smallmatrix}
 a  & b\\c& d 
\end{smallmatrix}\right)
 \in\PSL$, $M\neq I$. Hence both $\alpha$ and $\beta$ are fixed by $M$.  In particular, $\alpha= \frac{(a-d)+ \sqrt{(a+d)^2-4}}{2c},\beta=\balpha=\frac{(a-d)- \sqrt{(a+d)^2-4}}{2c}$.

In the opposite direction, suppose that $A\alpha^2+B\alpha+C=0$ with $\balpha <1,$ and $A,B,C\in\Z, A\geq 1, (A,B,C)=1$. The quadratic irrationals $\alpha$, $\balpha$, $-\alpha$, $\overline{-\alpha}=-\balpha$, and $\rho(\alpha)=\frac{1}{-\alpha+a_0}$ have the same discriminant. In fact, $M\alpha$ has the same discriminant as $\alpha$ for all $M\in\PSL$. We will show that for a given discriminant, there are only finitely many quadratic irrationals in $(1,2)$ with conjugate less that $1$. We do this by considering two cases. Let $D=B^2-4AC$. 
\begin{description}
 \item[$\bar\alpha<0$] In this case, we note that $\alpha-\balpha=\frac{\sqrt{D}}{A}>1.$ Thus, $\sqrt{D}>A$, and there are finitely many options for $A$. Similarly, $\alpha\balpha=\frac{C}{A}<0$ implies $C<0$. Thus, $B^2-4AC\geq B^2$, giving $|B|\leq \sqrt{D}$. Thus, there are only finitely many options for $B$. Finally, we have that $C=\frac{B^2-D}{2A}$ must also have finitely many options.
 
 \item[$0<\balpha<1$] Since $\balpha$ is positive, we now have that $\alpha\balpha=\frac{C}{A}>0$ implies $C>0$. We know that  \begin{align*}
 \bar\alpha=\frac{-B-\sqrt{D}}{2A}&<1<\alpha=\frac{-B+\sqrt{D}}{2A},\\
 -B-\sqrt{D}&<2A<-B+\sqrt{D}\\
 -\sqrt{D}&<2A+B<\sqrt{D}.
\end{align*}
This allows us to conclude that $ (2A+B)^2<D$. Thus  
\begin{align*}
4A^2+4AB+B^2&<B^2-4AC\\
A+B&<-C\\
A+B+C&<0.
\end{align*}

Now, we can use that $D-(2A+B)^2=-4A(A+B+C)>0$ to rewrite $A$ as:

\[A=\frac{-4A(A+B+C)}{-4(A+B+C)}=\frac{D-(2A+B)^2}{-4(A+B+C)}<D.\]
Thus there are finitely many options for $A$. Since $-\sqrt{D}<2A+B<\sqrt{D}$ and $C=\frac{B^2-D}{2A}$, there are also finitely many options for $B$ and $C$.
\end{description}

Now we consider the geodesic $\gamma$ with endpoints $\alpha$ and $\balpha$. We know that the endpoints $(\alpha_i,\balpha_i)$ of $\overline\rho^i(\gamma)$ are quadratic irrationals with the same discriminant as $\alpha=\alpha_0$ satisfying $1<|\alpha_i|<2$ and $\operatorname{sgn}{(\alpha_i)} \balpha_i<1$ for all $i\in\Z$. Thus, the sequence of $(\alpha_i, \balpha_i)_{i\in\Z}$ takes on finitely many values. Hence, there are some $i<j$ such that $\overline\rho^i(\gamma)=\overline\rho^j(\gamma)$. Since $\overline\rho$ is an invertible function, we have that $\gamma$ is a periodic point of $\overline{\rho}$.

By Theorem \ref{commuteleh}, we find that $(\alpha, \bar\alpha, 1)$ is a periodic point of $\widetilde{\mathcal{L}}$, and thus $(\alpha,-\bar\alpha)$ is a periodic point of $\mathcal{L}$. Therefore, $\alpha$ has a purely periodic Lehner expansion, and $\alpha, -\balpha$ had the desired form.
\end{proof}

Note that in the statement of the preceding proposition, the assumption that $\alpha,\balpha\neq 1$ is required to conclude that $\alpha$ is irrational and vice versa. Removing those assumptions gives $1=[\![\overline{(2,-1)}]\!]=-\langle\! \langle\overline{(-1/2)} \rangle\! \rangle$ are the solution to $2-\frac{1}{x}=x$ and $-x=\frac{-1}{2-x}$, ie, $x^2-2x+1=0$. 

Define the $m$-tail of a Lehner expansion $\alpha=[\![(a_0,\e_0)(a_1,\e_1)\ldots]\!]$ to be \[t_m(\alpha)=(-\e_0)(-\e_1)\cdots(-\e_{m})[\![(a_{m+1},\e_{m+1})(a_{m+2},\e_{m+2})\ldots]\!].\]

\begin{proposition}\label{lehner_tails}
Two irrationals $\alpha,\beta\in(1,2)$ are $\PSL$-equivalent if and only if there exist $r, s > 0$ such
that
$t_r(\alpha) = t_s(\beta)$. Note that $\alpha$ is $\PSL$-equivalent to all of its tails.
\end{proposition}

\begin{proof}
The proof follows closely the outline of statement 3.3.3 in \cite{Ser} and Proposition 6 in \cite{BM1}.
In one direction, if $\alpha$ and $\beta$ are tail equivalent, then $\alpha$ and $\beta$ are
$\PSL$-equivalent because $t_1(\alpha)=\frac{1}{a_0-\alpha}\in\R$. We can repeat this process to find $M_r,N_s\in \PSL$ such that $M_r\alpha=t_r(\alpha)=t_s(\beta)=N_s\beta$. Thus, $N^{-1}_s M_r \alpha=\beta$.

Conversely, suppose that $g\alpha=\beta$ for some $g\in\PSL$.
Fix $\delta<1$ and consider the geodesics $\gamma,\gamma^\prime \in {\mathcal A}$ with
$\gamma_{-\infty} =\gamma_{-\infty}^\prime =\delta$, $\gamma_\infty=\alpha$ and
$\gamma^\prime_\infty=\beta$. Their cutting sequences are
$\ldots \xi_\gamma A_1 A_2 \ldots$ and respectively $\ldots \xi_\gamma B_1 B_2 \ldots$
with $A_i,B_i$  either $L$ or $R$. The geodesics
$\gamma^{\prime\prime} =g\gamma$ and $\gamma^\prime$ have the same endpoint $\beta$.
Since their $\PSL$-cutting sequences in $L$ and $R$ coincide
(cf. \cite[Lemma 3.3.1]{Ser}), their cutting sequences also coincide. Thus,
the cutting sequence of $\gamma^{\prime\prime}$ is of the form
$\xi_{\gamma^{\prime\prime}} \ldots B_k B_{k+1}\ldots$ for some $k\geqslant 1$.
As $\gamma$ and $\gamma^{\prime\prime}$ are $\PSL$-equivalent geodesics, their cutting sequences
(after equivalent initial points) will coincide, implying that the cutting sequences of
$\gamma$ and $\gamma^\prime$ are of the form $\ldots \xi_\gamma A_1 \ldots A_r D_1 D_2 \ldots$
and $\ldots \xi_\gamma B_1 \ldots B_s D_1 D_2 \ldots$ respectively. Along with \eqref{rholeh}, the cutting sequences coinciding imply that $t_r(\alpha) = t_s(\beta)$.
\end{proof}

\begin{proposition}\label{leh_period}
 The Lehner expansion of an irrational $\alpha$ is eventually periodic if and only if $\alpha$ is a
quadratic irrational.
\end{proposition}
\begin{proof}
 The insertion and singularization algorithm implies that quadratic irrationals are eventually periodic \cite[Corollary 1]{DK2}.

In the other direction, if $\alpha$ is eventually periodic, then there exists $r\geq 0$ such that $t_r(\alpha)$ is purely periodic. Since $\alpha$ and $t_r(\alpha)$ are tail equivalent, we may find $g\in\PSL$ such that 
$g\alpha =t_r(\alpha)=\e [\![ \overline{(a_0,\e_0)\ldots (a_{r-1},\e_{r-1})} ]\!]$ for some $g\in\PSL$ and
$\e \in \{\pm 1\}$. Proposition \ref{lehner_pureperiod} gives that $g\alpha$ is a quadratic irrational, hence $\alpha$ is a quadratic irrational. 
\end{proof}

Similarly, the $m$-tail of a Farey expansion $\beta= \llangle(f_1/b_1)(f_2/b_2)\dots\rrangle$ is 
\[\tau_m(\beta)=(-f_1)\cdots(-f_{m})\llangle(f_{m+1}/b_{m+1})(f_{m+2}/b_{m+2})\dots\rrangle.\]
We may recover the same result from Propositions \ref{lehner_tails} and \ref{leh_period} for the Farey expansions considering $\gneg$ instead of $\gplus.$

\begin{cor}\label{farey_tails}
Two irrationals $\alpha,\beta\in(-\infty,1)$ are $\PSL$-equivalent if and only if there exist $r, s > 0$ such
that
$\tau_r(\alpha) = \tau_s(\beta)$. Again, $\alpha$ is $\PSL$-equivalent to all of its tails.
\end{cor}
\begin{proof}
In one direction, if $\alpha$ and $\beta$ are Farey-tail equivalent, then $\alpha$ and $\beta$ are $\PSL$-equivalent because $\tau_1(\alpha)=\frac{f_1b_1\alpha-1}{\alpha}\in\R$. We can repeat this process to find $M\in\PSL$ to move $\alpha$ to $\beta.$

In the other direction, we repeat a similar process from Proposition \ref{lehner_tails}. Here, we fix $\delta\in(1,2)$ and consider the geodesics $\gamma,\gamma'\in\AAA$ with $\gneg=\alpha, \gneg'=\beta$ and $\gplus=\gplus'=\delta.$ Now we have the cutting sequences $\dots A_2 A_1\x\dots$ and $\dots B_2 B_1\x\dots$ with $A_i, B_i$ either $L$ or $R.$ Finally, we repeat the argument of Proposition \ref{lehner_tails}.
\end{proof}

\begin{cor}\label{farey_period}
The Farey expansion of an irrational $x$ is eventually periodic if and only if $x$ is a quadratic irrational.
\end{cor}
\begin{proof}
 The insertion and singularization algorithm implies that quadratic irrationals are eventually periodic from Corollary \ref{farey_quad}. The other direction follows from applying the same argument as in Proposition \ref{leh_period}.
\end{proof}

\begin{proposition}\label{leh_closedgeo}
 A geodesic $\bar\gamma$ on $\MM$ is closed if and only if it has a lift $\gamma\in\AAA $ with purely periodic
endpoints $\gplus=\e[\![ \, \overline{(a_0,\e_0) \ldots (a_{r-1},\e_{r-1})} \,]\!]$ and $\gneg=-\e\llangle \,\overline{(\e_{r-1}/a_{r-1}) \ldots (\e_0/a_0)} \,\rrangle<1$ for some $\e\in\{\pm1\}$ and $(-\e_0)\cdots(-\e_{r-1})= 1$.
\end{proposition}

\begin{proof} We will use $\bar{P}$ to denote the map induced on $X$ by the map $\bar{\rho}$ on $\SSS$ defined in Section \ref{surface}. Recall that for $z\in\gamma\in\AAA$, $\rho(z)=\frac{1}{\operatorname{sign}(\gplus) a_0(\gplus)-z}$.

A geodesic $\bar\gamma$ is closed on $\MM$ if and only if there exists $r$ such that $\bar P^r(\x,u_\gamma)=(\x,u_\gamma)$ for $(\x,u_\gamma)\in X $. 

In one direction, from Lemma \ref{lehner_rho_period}, we know that if $\gplus=\e[\![ \, \overline{(a_0,\e_0) \ldots (a_{r-1},\e_{r-1})} \,]\!]$ and \\$\gneg=-\e\llangle \,\overline{(\e_{r-1}/a_{r-1}) \ldots (\e_0/a_0)} \,\rrangle<1$ for some $\e\in\{\pm1\}$ and $(-\e_0)\dots(-\e_{r-1})= 1$, then
$\bar\rho\ ^r(\gplus,\gneg)=(\gplus,\gneg)$. Thus, $\bar P^{r}(\x,u_\gamma)=(\x,u_\gamma)$, and $\rho^{r}(\x)$ is also a base point for $u_\gamma$. 

In the other direction, we assume that there exists some $r$ such that $\rho^r(\x)=\x$. Since $\x$ is determined by $(\gplus, \gneg)=(\e[\![(a_0,\e_0)(a_1,\e_1)\dots]\!],-\e\llangle(a_{-1},\e_{-1})(a_{-2},\e_{-2})\dots\rrangle)$, we also have that 
\begin{align*}\bar\rho\ ^r(\gplus, \gneg)&= \epsilon(-\e_0)\dots(-\e_{r-1})([\![(a_r,\e_r)(a_{r+1},\e_{r+1})\dots]\!],\llangle(\e_{r-1}/a_{r-1})\dots(a_{-1},\e_{-1})\dots\rrangle)
\\&=(\gplus,\gneg)=(\e[\![(a_0,\e_0)(a_1,\e_1)\dots]\!],-\e\llangle(\e_{-1}/a_{-1})(\e_{-2}/a_{-2})\dots\rrangle).\end{align*}
Thus, we find $(-\e_0)\cdots(-\e_{r-1})=1$ and $(a_i,\e_i)=(a_{i+r},\e_{i+r})$ for all $i\in\Z$.
\end{proof}

Using the fact that $\rho $ preserves lengths, we find that $d(\x,\y)=d(\rho(\x),\rho(\y))$. If we let $\rho(\y)=x+iy$, then 
\[\frac{|\rho(\gneg)-\rho(\y)|}{|\rho(\gplus)-\rho(\y)|}=\sqrt{\frac{x-\rho(\gneg)}{\rho(\gplus)-x}}, \qquad \frac{|\rho(\gneg)-\rho(\x)|}{|\rho(\gplus)-\rho(\x)|}=\sqrt{\frac{1-\rho(\gneg)}{\rho(\gplus)-1}}.
\]
Since $\rho(\x)$ lies on the geodesic  $\operatorname{sign}(x)[1,2]$, we find
$\left|x+iy-\operatorname{sign}(x)\frac{3}{2}\right|=\frac{1}{2}$. We also have that $\rho(\y)$ lies on the geodesic $
[\rho(\gplus),\rho(\gneg)]$ and $|x+iy-\frac{1}{2}(\gplus+\gneg)|=\frac{1}{2}(\gplus-\gneg)$. Thus, 
\[x=\operatorname{Re}\rho(\y)=\frac{2-\rho(\gplus)\rho(\gneg)}{3\operatorname{sign}(x)-\rho(\gplus)-\rho(\gneg)}.
\]

Note that $\operatorname{sign}(x)=\operatorname{sign}(\rho(\gplus))=-\e\e_0(\gplus)$. We find that 
\begin{equation}
 d(\x,\y)=\frac{1}{2}\log\left(\frac{\big(\rho(\gplus)+\e\e_0(\gplus)\big)\big(\rho(\gplus)+2\e\e_0(\gplus)\big)\big(1-\rho(\gneg)\big)}{\big(\rho(\gneg)+\e\e_0(\gneg)\big)\big(\rho(\gneg)+2\e\e_0(\gneg)\big)\big(1-\rho(\gplus)\big)}\right).
\end{equation}

\section{An alternate dual expansion to the Farey expansions}\label{alt_leh}
Schweiger \cite{Sch06} describes continued factions and their dual expansions using M\"obius transformations. The goal of this section is to define an alternate dual expansion of the Farey expansions that matches with Schweiger's definition of a \textit{natural dual}, define an alternate natural extension of the Farey expansions using the natural dual, and prove that the measure in Theorems \ref{form_natext_measure} and \ref{natmeasure} is invariant under the alternate natural extension. By doing this, we give a pair of dual continued fraction extensions that more closely match other dual pairs of continued fractions such as the even and odd continued fractions \cite{Sch1}, the $\alpha$-continued fractions \cite{Nak}, the Rosen continued fractions \cite{Ros}, and the $\alpha$-odd continued fractions \cite{BM2}.  We also change some of the notation from \cite{Sch06} to match the standard notation for M\"obius transformations.

\begin{definition}\cite[Definition 1]{Sch06} Let $B$ be an interval and $T: B\to B$ be a map. We call $(B,T)$ a \textit{M\"obius system} if there exists a countable sequence of intervals $(J_k), k\in I$ and an associated sequence of matrices
\[\alpha(k)=\begin{pmatrix}a_k&b_k\\c_k & d_k\end{pmatrix},\quad \det(\alpha(k))\neq 0\]
where
\begin{itemize}
    \item $\{\overline{J_k}\}$ partition $B,$
    \item $Tx=\frac{a_k x+b_k}{c_k x+d_k}$ for $x\in J_k,$
    \item $T|_{J_k}$ is a bijection from $J_k$ onto $B$.
\end{itemize}
\end{definition}
\noindent For generic continued fraction expansions, the countable partition is the set of intervals $B(a_1,\e_1)$, where the first digit of $x\in B(a_1,\e_1)$ is $(a_1,\e_1)$. 

Recall that the Gauss map for the Farey expansions is 
\[F(x):=
\begin{cases}
 \frac{-1}{x} -2&\textnormal{if }  x\in\left[-1,0\right)\\
 0 &\textnormal{if }  x=0\\
  \frac{1}{x}-1 &\textnormal{if }  x\in\left(0,\infty\right).
\end{cases}
\] 
In our case, we find for the Farey expansions, $B=(-1,\infty)$ with $B(2,-1)=(-1,0)$ and $B(1,+1)=(0,\infty).$ The associated matrices are \[M_{B(2,-1)}:=\begin{pmatrix}-2&-1 \\1&0\end{pmatrix},\quad\textnormal{for }B(2,-1)=[-1,0)\] and  \[M_{B(1,+1)}:=\begin{pmatrix}-1&1 \\1&0\end{pmatrix},\quad\textnormal{for }B(1,+1)=(0,\infty).\]

We also have a M\"obius system for the Lehner expansions. We need to use different notation for these expansions to distinguish this system from the one that generates the Farey expansions. Let $J_1=[\frac{3}{2},2)$ and $J_2=[1,\frac{3}{2}).$ Then  $\alpha(1)=\begin{pmatrix}0&1\\1&-1\end{pmatrix}$ and $\alpha(2)=\begin{pmatrix}0&1\\-1&2\end{pmatrix}$.

\begin{definition}\cite[Definition 2]{Sch06}
A M\"obius system $(B^*,T^*)$ is called a \textit{natural dual} of $(B,T)$ if there is a partition $\{J_k^*\},k\in I$ of $B^*$ such that $T^*y=\frac{a_k y+c_k}{b_k y +d_k}$. That is, $\alpha^*(k)$ is the transpose of $\alpha(k)$.

For continued fraction expansions as M\"obius systems, we write the partition as $\{B^*(a_1,e_1)\}$ and the corresponding matrices as $N_{B^*(a_1,e_1)}$. 
\end{definition}

The Lehner expansions are not the natural dual system to the Farey fractions. However, conjugating the $\alpha(k)$ by $\left(\begin{smallmatrix}
 0&1\\1&0
\end{smallmatrix}\right)$ gives 
\[N_{B^*(2,-1)}=M^T_{B(2,-1)}=\begin{pmatrix}2&-1 \\1&0\end{pmatrix},\quad\textnormal{for }B^*(2,-1)=\left(\frac{1}{2},\frac{2}{3}\right]\] and  \[N_{B^*(1,+1)}=M^T_{B(1,+1)}=\begin{pmatrix}-1&1 \\1&0\end{pmatrix},\quad\textnormal{for }B^*(1,+1)=\left(\frac{2}{3},1\right].\]
Thus, we define the natural dual expansions on $(\frac{1}{2},1]$ 
\[x=\cfrac{1}{a_0+\cfrac{\e_0}{a_1+\cfrac{\e_1}{\dots}}}=\frac{1}{[\![(a_0,e_0)(a_1,e_1)\dots]\!]}\]
The Gauss map for this continued fraction expansion $F^*:(\frac{1}{2},1]\to(\frac{1}{2},1]$ is given by:
\begin{equation*}
 F^*(x)=
\begin{cases}
\frac{1}{x}-1  &\textnormal{if } x\in(\frac{1}{2},\frac{2}{3}],\\
 \frac{-1}{x}+2 &\textnormal{if } x\in(\frac{2}{3},1].
\end{cases}
\end{equation*}
 By construction, $F^*$ is conjugate to $L$ by the map $x\mapsto\frac{1}{x}$ and to the Farey map $\tau$  in Section \ref{leh} by $x\mapsto\frac{1}{x+1}$. We find that $(a_i,\e_i)=
\begin{cases}
(1,+1) & \textnormal{if } (F^*)^{i}(x)\in(\frac{1}{2},\frac{2}{3}],\\
 (2,-1) & \textnormal{if } (F^*)^{i}(x)\in(\frac{2}{3},1],
\end{cases}$
agreeing with the fact that the Lehner expansion of $\frac{1}{x}$ comes from $L^{i}(\frac{1}{x})$. Since $F^*$ is conjugate to $L$, we can find an infinite $F^*$ invariant measure by pushing forward the invariant density $\frac{dx}{x+1}$ to get:
\begin{proposition}
 The infinite measure $\frac{dx}{x(1-x)}$ is $F^*$-invariant.
\end{proposition}
\noindent Since the regular continued fraction expansion of numbers in $[\frac{1}{2},1]$ begins with $n_1=1$, we can use the same insertion and singularization algorithms as \eqref{leh_insert}.

\begin{theorem}\cite[Theorem 1 rephrased]{Sch06}\label{form_natext_measure}
The map defined by dual pairs $\bar T: B\times B^*\to B\times B^*$ for a Gauss map $T: B\to B$ is given by $\bar{T}(x,y)=(M_{(a_i,e_i)} x, N_{(a_i,e_i)}^{-1} y)$ for $x\in B(a_i,e_i)$. The measure  $\frac{dxdy}{(1+xy)^2}$ is $\bar T$-invariant .
\end{theorem}

The measure $\frac{dxdy}{(1+xy)^2}$ is invariant for the natural extension of the Gauss map of many other types of continued fractions, such as the even and odd continued fractions \cite{Sch1}, the $\alpha$-continued fractions \cite{Nak}, the Rosen continued fractions \cite{Ros}, and the $\alpha$-odd continued fractions \cite{BM2}. By constructing the natural dual pair, we will also construct a natural extension to $F$ with this invariant measure.

In order to prove Theorem \ref{form_natext_measure} in this setting, we construct the natural extension of this alternate expansion, $\bar F:(\frac{1}{2},1]\times[-1,\infty) \to (\frac{1}{2},1]\times[-1,\infty)$ given by:
\begin{equation}
 \bar F(x,y)=\left(\e_0\left(\frac{1}{x}-a_0\right),\frac{\e_0}{a_0+y}\right)=
\begin{cases}
(\frac{1}{x}-1,\frac{1}{1+y})&\textnormal{if }x\in(\frac{1}{2},\frac{2}{3}],\\
(\frac{-1}{x}+2,\frac{-1}{2+y})&\textnormal{if }x\in(\frac{2}{3},1].
\end{cases}
\end{equation}
Again, we see that $\bar F$ is conjugate to $\mathcal{L}$ by $(\frac{1}{x},y)$. Since $\bar{F}$ is conjugate to $\mathcal{L}$ by an isomorphism, it is also a natural extension of $F$. We could recover the $\bar{F}$-invariant measure by pushing forward the $\mathcal{L}$-invariant measure. Below, we repeat the construction of Theorem \ref{commuteleh} and Corollary \ref{Lmeasure}

For the geodesic coding, we define $J^*:\SSS \to(\frac{1}{2},1]\times[-1,\infty)\times\{\pm1\}$ by \[J^*(x,y)=\operatorname{sign}(x)\left(\frac{1}{x},-y,1\right)= \begin{cases}
(\frac{1}{x},-y,1) & \textnormal{ if } x\in\left(1,2\right),\ y<1, \\
(\frac{-1}{x},y,-1) & \textnormal{ if } x\in\left(-2,-1\right),\ y>-1,
\end{cases}\]
and $\tilde F: (\frac{1}{2},1]\times[-1,\infty)\times\{\pm1\}\to (\frac{1}{2},1]\times[-1,\infty)\times\{\pm1\}$ by $\tilde F(x,y,\e)=(\bar F(x,y),-\e_0(x)\e)$.
As in Theorem \ref{commuteleh}, we find that $J^*\circ\bar\rho=\tilde F\circ J^*$. Finally, we recover the result from Theorem \ref{form_natext_measure} that:
 
\begin{theorem}\label{natmeasure}
 The infinite measure $\frac{dxdy}{(1+xy)^2}$ is $\bar F$-invariant.
\end{theorem}

\section{Acknowledgements}

The first draft of this work was completed while the author was at the University of Illinois at Urbana-Champaign. The research was partially funded by Nate Snyder's  NSF CAREER Grant DMS-1454767.

The author is also grateful to Florin Boca for useful comments on the drafts of the paper. Also to the referee for careful reading and a number of useful comments and corrections that improved the paper, especially Proposition \ref{lehner_pureperiod} and Section \ref{alt_leh}.
\bibliographystyle{alpha}
\bibliography{lehner_coding_arxiv}

\begin{thebibliography}{Moe82}

\bibitem[Arn94]{Arn}
Pierre Arnoux.
\newblock Le codage du flot g{\'e}od{\'e}sique sur la surface modulaire.
\newblock {\em Enseign. Math}, 40(1-2):29--48, 1994.

\bibitem[BM18]{BM1}
Florin~P. Boca and Claire Merriman.
\newblock Coding of geodesics on some modular surfaces and applications to odd
  and even continued fractions.
\newblock {\em Indag. Math.}, 29(5):1214 -- 1234, 2018.

\bibitem[BM19]{BM2}
Florin~P. Boca and Claire Merriman.
\newblock {$\alpha$}-{E}xpansions with odd partial quotients.
\newblock {\em J. Number Theory}, 199:322--341, 2019.

\bibitem[DK00]{DK2}
Karma Dajani and Cor Kraaikamp.
\newblock `{T}he mother of all continued fractions''.
\newblock {\em Colloq. Math.}, 84/85(part 1):109--123, 2000.
\newblock Dedicated to the memory of Anzelm Iwanik.

\bibitem[Hee19]{Hee}
Byron Heersink.
\newblock Distribution of the {P}eriodic {P}oints of the {F}arey {M}ap.
\newblock {\em Comm. Math. Phys.}, 365(3):971--1003, 2019.

\bibitem[Hop39]{Hopf}
E~Hopf.
\newblock Ergodentheorie (berlin, 1937). e. hopf, ber.
\newblock {\em Verhandl. Sachs. Akad. Viss}, 91:261, 1939.

\bibitem[Ito89]{Ito}
Shunji Ito.
\newblock Algorithms with mediant convergents and their metrical theory.
\newblock {\em Osaka J. Math.}, 26(3):557--578, 1989.

\bibitem[Kra91]{Kr}
Cor Kraaikamp.
\newblock A new class of continued fraction expansions.
\newblock {\em Acta Arith.}, 57(1):1--39, 1991.

\bibitem[KU07]{KU1}
Svetlana Katok and Ilie Ugarcovici.
\newblock Symbolic dynamics for the modular surface and beyond.
\newblock {\em Bull. Amer. Math. Soc.}, 44(1):87--132, 2007.

\bibitem[Leh94]{Leh}
Joseph Lehner.
\newblock Semiregular continued fractions whose partial denominators are 1 or
  2.
\newblock In {\em The mathematical legacy of {W}ilhelm {M}agnus: groups,
  geometry and special functions ({B}rooklyn, {NY}, 1992)}, volume 169 of {\em
  Contemp. Math.}, pages 407--407. Amer. Math. Soc., Providence, RI, 1994.

\bibitem[Moe82]{Moe}
Richard Moeckel.
\newblock Geodesics on modular surfaces and continued fractions.
\newblock {\em Ergodic Theory Dynam. Systems}, 2(1):69--83, 1982.

\bibitem[Nak81]{Nak}
Hitoshi Nakada.
\newblock Metrical theory for a class of continued fraction transformations and
  their natural extensions.
\newblock {\em Tokyo J. Math.}, 4(2):399--426, 1981.

\bibitem[Ros54]{Ros}
David Rosen.
\newblock A class of continued fractions associated with certain properly
  discontinuous groups.
\newblock {\em Duke Mathematical Journal}, 21(3):549--563, 1954.

\bibitem[Sch82]{Sch1}
Fritz Schweiger.
\newblock Continued fractions with odd and even partial quotients.
\newblock {\em Arbeitsber. Math. Inst. Univ. Salzburg}, 4:59--70, 1982.

\bibitem[Sch06]{Sch06}
Fritz Schweiger.
\newblock {\em Differentiable equivalence of fractional linear maps}, volume
  Volume 48 of {\em Lecture Notes--Monograph Series}, pages 237--247.
\newblock Institute of Mathematical Statistics, Beachwood, Ohio, USA, 2006.

\bibitem[Ser85]{Ser}
Caroline Series.
\newblock The modular surface and continued fractions.
\newblock {\em J. London Math. Soc.}, 2(1):69--80, 1985.

\end{thebibliography}


\end{document}